\documentclass[letterpaper, 10 pt, conference]{ieeeconf}  

\usepackage{amsmath}
\usepackage{amsfonts}
\usepackage{amssymb}
\usepackage{verbatim}
\usepackage{latexsym}
\usepackage{color}
\usepackage{algorithm2e}
\usepackage{diagrams}
\usepackage{tikz}

\usepackage{verbatim}

\usepackage{latexsym}

\renewcommand{\rm}{\textcolor{red}}

\newcommand{\Ac}{\mathrm{Ac}}

\newtheorem{proposition}{Proposition}
\newtheorem{assumption}{Assumption}
\newtheorem{theorem}{Theorem}

\newtheorem{definition}{Definition}

\newtheorem{remark}{Remark}

\begin{document}

\IEEEoverridecommandlockouts                              



\overrideIEEEmargins










\title{\LARGE \bf On Approximate Diagnosability of Nonlinear Systems
}

\author{Elena De Santis, Giordano Pola and Maria Domenica Di Benedetto
\thanks{The authors are with the Department of Information Engineering, Computer Science and Mathematics, Center of Excellence DEWS,
University of L{'}Aquila, 67100 L{'}Aquila, Italy, \{elena.desantis,giordano.pola,mariadomenica.dibenedetto\}@univaq.it.
}
\thanks{This work has been partially supported by the Center of Excellence for Research DEWS, University of L'Aquila, Italy.}
}

\maketitle

\thispagestyle{empty}

\pagestyle{empty}


\begin{abstract}
This paper deals with diagnosability of discrete--time nonlinear systems with unknown inputs and quantized outputs. We propose a novel notion of diagnosability that we term approximate diagnosability, corresponding to the possibility of detecting within a finite delay and
\textit{within a given accuracy} if a set of faulty states is reached or not. Addressing diagnosability in an approximate sense is primarily motivated by the fact that system outputs in concrete applications are measured by sensors that introduce measurement errors.
Consequently, it is not possible to detect exactly if the state of the system has reached or not the set of faulty states.
In order to check approximate diagnosability on the class of nonlinear systems we use tools from formal methods.
We first derive a symbolic model approximating the original system within any desired accuracy. This step allows us to check approximate diagnosability of the symbolic model. We then establish the relation between approximate diagnosability of the symbolic model and of the original nonlinear system.
\end{abstract}

\IEEEoverridecommandlockouts

\thispagestyle{empty}

\pagestyle{empty}

\vspace{3mm}
\textbf{keywords: } approximate diagnosability, nonlinear systems, quantized systems, symbolic models.

\section{Introduction} \label{sec1}

The increasing complexity in real--world engineered systems requires great attention to performance degradation, safety hazards and occurrence of faults, which must be detected as soon as possible to possibly restore nominal behavior of the system. The notion of diagnosability plays a key role in this regard, since it corresponds to the possibility of detecting within a finite delay if a fault, or in general a hazardous situation, occurred. This notion has been extensively studied  both  in the Discrete--Event Systems (DES) community and control systems community, and the related literature is very broad. Within the DES community, after the seminal work \cite{Sampath1995}, several results have been achieved, see e.g. \cite{Wang:Aut10,Wang:11,Debouk02,Su05,Zhou08,Schmidt13,Wonham:03,Schuppen} and the references therein. An excellent review of recent advances on diagnosis methods can be found in \cite{Zaytoon13}. More recently, a unifying framework for the study of observability and diagnosability of DES has been also proposed in \cite{DeSDiB:automatica17}. Within the control systems community, for fault--tolerant control, an early review paper was presented in \cite{Stengel1991}, which introduced the basic concepts of fault--tolerant control and analyzed the applicability of artificial intelligence to fault--tolerant control systems. Subsequent overviews appeared in \cite{Blanke1997,Patton1997}. Reconfigurable fault-tolerant control systems are reviewed in \cite{Jiang2005,Lunze2008,Zhang2008} and some results on fault--tolerant control for nonlinear systems in \cite{Benosman2010}. A recent comprehensive survey on diagnosability of continuous systems is reported in \cite{Gao2015}. Extensions to hybrid systems, featuring both discrete and continuous dynamics, are also present in the literature. For example, \cite{Tripakis2002} addressed diagnosability for timed automata, \cite{Bayoudh2008} diagnosability for hybrid systems, \cite{Bayoudh2006,Dib2011,Deng2016} propose abstraction techniques for diagnosability of hybrid automata. Apart from differences in the class of systems considered and in the way faults are modeled, to the best of our knowledge, existing papers, except for \cite{Lunze2010,DePersis2013}, \textit{either assume that state variables are available, or assume the exact knowledge of output variables}. This is rather limiting in concrete applications where state variables cannot be directly measured, or output variables are measured by sensors that introduce measurement errors. \\
The papers \cite{Lunze2010}  and \cite{DePersis2013} study diagnosability for quantized systems. They both model faults as additional inputs to the system. The former considers continuous--time nonlinear systems and the detection is done in a stochastic setting, by assuming an appropriate description of the occurrence of faults. The latter analyzes discrete--time linear systems and the faults are detected, provided that they belongs to an appropriate class of functions. \\
In this paper, we introduce a new notion of diagnosability, that we term \textit{approximate diagnosability}, for discrete--time nonlinear systems with unknown inputs and quantized output measurements.
Given an accuracy $\rho\geq 0$ and a set of faulty states $\mathcal{F}$, approximate diagnosability corresponds to the possibility of detecting, within a finite time delay:
\begin{itemize}
\item if the system's state reached the set $\mathcal{F}+\mathcal{B}_{\rho}(0)$, obtained by adding to $\mathcal{F}$ a closed ball $\mathcal{B}_{\rho}(0)$ centered at the origin and with radius $\rho$, and 
\item if the system's state has never reached the set $\mathcal{F}$.
\end{itemize}
This ambiguity around the set $\mathcal{F}$ reflects uncertainties introduced by measurement errors. When the accuracy $\rho=0$, approximate diagnosability translates to dynamical systems the notion of diagnosability investigated in \cite{DeSDiB:automatica17} for DES. \\
In order to check this property on the class of nonlinear systems we use tools from formal methods. Under an assumption of incremental stability of the nonlinear system we first derive a symbolic model approximating the original system within any desired accuracy. We recall that a symbolic model is an abstract description of the control system where each state corresponds to an aggregate of continuous states and each control label corresponds to an aggregate of continuous inputs. We then extend the classical notion of diagnosability given for DES to metric symbolic models and in an approximate sense. Algorithms for checking this property can be easily obtained by naturally extending those proposed in
\cite{DeSDiB:automatica17} to an approximate sense. We finally show how to check approximate diagnosability of the original nonlinear system by analyzing the same property on the obtained symbolic model. Computational complexity of the proposed approach is also discussed.\\
This paper is organized as follows. Section \ref{sec2} introduces notation and preliminary definitions. Section \ref{sec3} introduces the notion of approximate diagnosability for the class of discrete--time nonlinear systems. In Section \ref{sec4} we first derive symbolic models approximating the nonlinear system in the sense of approximate bisimulation, we then extend the notion of diagnosability given for DES to metric symbolic systems and in an approximate sense, and then establish connections between approximate diagnosability of the symbolic model and approximate diagnosability of the original nonlinear system. Some concluding remarks are offered in Section \ref{sec5}.

\section{Notation and preliminary definitions}\label{sec2}
The symbols $\mathbb{N}$, $\mathbb{Z}$, $\mathbb{R}$, $\mathbb{R}^{+}$ and $\mathbb{R}_{0}^{+}$ denote the set of nonnegative integer, integer, real, positive real, and nonnegative real numbers, respectively. The symbol $0_n$ denotes the origin in $\mathbb{R}^n$. Given $a,b\in \mathbb{Z}$, we denote $[a;b]=[a,b]\cap\mathbb{Z}$.
Given a set $X$, the symbol $2^X$ denotes the power set of $X$. Given a pair of sets $X$ and $Y$ and a relation $\mathcal{R}\subseteq X\times Y$, the symbol $\mathcal{R}^{-1}$ denotes the inverse relation of $\mathcal{R}$, i.e.
$\mathcal{R}^{-1}=\{(y,x)\in Y\times X:( x,y)\in \mathcal{R}\}$. Given $X'\subseteq X$ and $Y'\subseteq Y$, we denote $\mathcal{R}(X')=\{y\in Y | \exists x\in X' \text{ s.t. } (x,y)\in \mathcal{R}\}$ and $\mathcal{R}^{-1}(Y')=\{x\in X | \exists y\in Y' \text{ s.t. }  (x,y)\in \mathcal{R}\}$.
Given a function $f:X\rightarrow Y$ and $X'\subseteq X$ the symbol $f(X')$ denotes the image of $X'$ through $f$, i.e. $f(X')=\{y\in Y | \exists x\in X' \text{ s.t. } y=f(x)\}$ and the symbol $f|_{X'}$ denotes the restriction of $f$ to $X'$ that is $f|_{X'}:X'\rightarrow Y$ such that $f|_{X'}(x')=f(x')$ for all $x'\in X'$. A continuous function $\gamma:\mathbb{R}_{0}^{+}
\rightarrow\mathbb{R}_{0}^{+}$, is said to belong to class $\mathcal{K}$ if it
is strictly increasing and \mbox{$\gamma(0)=0$}; $\gamma$ is said to belong to class
$\mathcal{K}_{\infty}$ if \mbox{$\gamma\in\mathcal{K}$} and $\gamma(r)\rightarrow
\infty$ as $r\rightarrow\infty$. A continuous function \mbox{$\beta:\mathbb{R}_{0}^{+}\times\mathbb{R}_{0}^{+}\rightarrow\mathbb{R}_{0}^{+}$} is said to belong to class $\mathcal{KL}$ if for each fixed $s$, the map $\beta(r,s)$ belongs to class $\mathcal{K}_{\infty}$ with respect to $r$ and, for each fixed $r$, the map $\beta(r,s)$ is decreasing with respect to $s$ and $\beta(r,s)\rightarrow0$ as \mbox{$s\rightarrow\infty$}.
Symbol $I_r$ denotes the identity matrix in $\mathbb{R}^r$. Given a vector $x\in\mathbb{R}^{n}$ we denote by $x(i)$ the $i$--th element of $x$ and by $\Vert x\Vert$ the infinity norm of $x$. Given $a\in\mathbb{R}$ and $X\subseteq \mathbb{R}^{n}$, the symbol $aX$ denotes the set $\{y\in\mathbb{R}^{n}| \exists x\in X \text{ s.t. } y=ax\}$.
Given $\theta\in\mathbb{R}^+$ and $x\in\mathbb{R}^n$, we denote
\[
\begin{array}
{l}
\mathcal{B}_{\theta}(x)=\{y\in\mathbb{R}^{n}| \Vert x-y \Vert \leq \theta\};\\
\mathcal{B}_{[-\theta,\theta[}^n(x)=
\left\{
\begin{array}
{l}
y\in\mathbb{R}^{n}| \\
y(i) \in [-\theta+x(i),\theta+x(i)[, i\in[1;n]
\end{array}
\right\}.
\end{array}
\]
Note that for any $\theta\in\mathbb{R}^+$, the collection of $\mathcal{B}^n_{[-\theta,\theta[}(x)$ with $x$  ranging in $2\theta \, \mathbb{Z}^n$ is a partition of $\mathbb{R}^n$. Given a set $X\subseteq \mathbb{R}^n$ we denote by $\mathcal{B}_{\theta}(X)$ the set $\bigcup_{x\in X}\mathcal{B}_{\theta}(x)$. We now define the quantization function.

\begin{definition}
Given a positive $n\in\mathbb{N}$ and a quantization parameter $\theta\in\mathbb{R}^+$, the quantizer in $\mathbb{R}^n$ with accuracy $\theta$ is a function
\[
[\,\cdot \, ]_{\theta}^n:\mathbb{R}^n \rightarrow 2\theta \mathbb{Z}^n,
\]
associating to any $x\in  \mathbb{R}^{n}$ the unique vector $[x]^n_{\theta} \in 2\theta \mathbb{Z}^{n}$ such that $x\in \mathcal{B}^n_{[-\theta,\theta[}([x]^n_{\theta})$.
\end{definition}

Definition above naturally extends to sets $X\subseteq \mathbb{R}^{n}$ when $[X]^n_{\theta}$ is interpreted as the image of $X$ through function $[\, \cdot \, ]^n_{\theta}$.

\section{Nonlinear systems and approximate diagnosability} \label{sec3}
\begin{figure}[t]
\label{fig1}
\begin{center}
\includegraphics[scale=0.7]{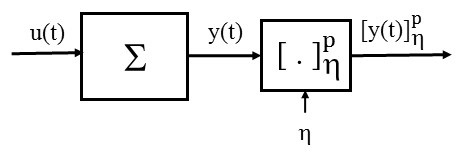}
\caption{Nonlinear system with quantized output measurements.}
\end{center}
\end{figure}

The class of nonlinear systems that we consider in this paper is described by
\begin{equation}
\label{plant}
\Sigma :
\left\{
\begin{array}
{l}
x(t+1)=f(x(t),u(t)),\\
y(t)=\left[
\begin{array}
{cc}
I_p & 0
\end{array}
\right]x(t),\\
x(0)\in \mathcal{X}_0,
x(t) \in \mathbb{R}^n ,u(t) \in U, y(t) \in \mathbb{R}^p, t\in \mathbb{N},
\end{array}
\right.
\end{equation}
where:
\begin{itemize}
\item $x(t)$, $u(t)$ and $y(t)$ denote, respectively, the state, the input and the output, at time $t \in\mathbb{N}$;
\item $\mathbb{R}^{n}$ is the state space;
\item $\mathcal{X}_0 \subseteq \mathbb{R}^{n}$ is the set of initial states;
\item $U\subseteq\mathbb{R}^m$ is the input set;
\item $\mathbb{R}^p$ is the output space with $p<n$;
\item \mbox{$f:\mathbb{R}^{n}\times \mathbb{R}^m \rightarrow \mathbb{R}^{n}$} is the vector field.
\end{itemize}

\noindent In this paper we make the following:
\begin{assumption}
\label{A0}
\textit{ }
\begin{itemize}
\item set $\mathcal{X}_0$ is compact;
\item set $U$ is compact and contains the origin $0_{m}$;
\item function $f$ is continuous in its arguments and satisfies $f(0_n,0_m)=0_n$.
\end{itemize}
\end{assumption}

We denote by $\mathcal{U}$ the collection of input functions from $\mathbb{N}$ to $U$. Given $t_f\in\mathbb{N}$, a function
\begin{equation}
\label{xtraj}
x:[0;t_f]\rightarrow\mathbb{R}^{n}
\end{equation}
is said to be a \textit{state trajectory} of $\Sigma$ if $x(0)\in \mathcal{X}_0$ and there exists $u\in\mathcal{U}$ satisfying \mbox{$x(t+1)=f\left(x(t),u(t)\right)$}, for all $t\in [0;t_f]$. Given $t_f\in\mathbb{N}$, a function \mbox{$y:[0;t_f]\rightarrow\mathbb{R}^{p}$} is said to be an \textit{output trajectory} of $\Sigma$ if there exists a state trajectory \mbox{$x:[0;t_f]\rightarrow\mathbb{R}^{n}$} of $\Sigma$ such that $y(t)=\left[
\begin{array}
{cc}
I_p & 0
\end{array}
\right]x(t)$ for all times $t\in [0;t_f]$.
Given $t_f\in\mathbb{N}$, we denote by $\mathcal{Y}_{t_f}$ the collection of output trajectories of $\Sigma$ with domain $[0;t_f]$.

\begin{remark}
The choice in the output function in nonlinear system $\Sigma$ is motivated by many concrete applications where the output variables correspond to a selection of the state variables. However, it is easy to see that there is no loss of generality in considering output function in (\ref{plant}) in the form of $y(t)=Cx(t)$, $t\in\mathbb{N}$, for some matrix $C \in \mathbb{R}^{p\times n}$, instead of
$y(t)=
\left[
\begin{array}
{cc}
I_p & 0
\end{array}
\right]x(t)$, $t\in\mathbb{N}$. General nonlinear output functions are not considered in this paper and will be the object of future investigations.
\end{remark}

Throughout the paper we will make the following

\begin{assumption}
Inputs $u(.)\in \mathcal{U}$ of $\Sigma$ are not known.
\end{assumption}

\begin{assumption}
Output $y(t)$ of $\Sigma$ at time $t\in\mathbb{N}$ is only available through its quantization $[y(t)]^p_{\eta}$, where $\eta\in\mathbb{R}^+$ is the quantization parameter, see Fig. 1.

\end{assumption}

Given the quantization parameter $\eta\in\mathbb{R}^+$ and $t_f\in\mathbb{N}$, we denote by $\mathcal{Y}_{t_f}^{\eta}$ the collection of quantized output trajectories generated by $\Sigma$ with domain $[0;t_f]$, i.e. the collection of functions
\[
y_{t_f,\eta}:[0;t_f]\rightarrow [\mathbb{R}^{p}]^p_{\eta},
\]
such that there exists $y_{t_f} \in \mathcal{Y}_{t_f}$ for which:
\[
y_{t_f,\eta} (t)=
[ y_{t_f} (t) ]_{\eta}^p,\forall t\in [0;t_f].
\]
We also set
\[
\mathcal{Y}^{\eta}=\bigcup_{t_f \in \mathbb{N}} \mathcal{Y}_{t_f}^{\eta},
\]
corresponding to the collection of all quantized output trajectories of $\Sigma$.
We can now propose the notion of approximate diagnosability for nonlinear systems.

\begin{definition}
(Approximate diagnosability of nonlinear systems) Given a desired accuracy $\rho\in\mathbb{R}^{+}_0$
and a set of faulty states $\mathcal{F}\subseteq\mathbb{R}^{n}$ with
\[
\mathcal{F}\cap\mathcal{X}_{0}=\varnothing,
\]
nonlinear system $\Sigma$ in (\ref{plant}) is $(\rho,\mathcal{F})$--diagnosable if there exists a finite delay
$\Delta\in\mathbb{N}$, and
\[
\mathcal{D}:\mathcal{Y}^{\eta}\rightarrow\left\{  0,1\right\},
\]
called the diagnoser, such that
\[
\mathcal{D}\left(  y_{0,\eta}\right)  =0,
\]
and whenever for some time $\mathbf{t}>0$
\[
\left(  x(\mathbf{t})\in\mathcal{F}\right)  \wedge\left(  x(t)\notin\mathcal{F},\forall
t\in\lbrack0;\mathbf{t}-1]\right)
\]
we have
\[
\mathcal{D}\left(  y_{\mathbf{t}+\Delta,\eta}\right)  =1.
\]
Conversely, whenever for some time $\mathbf{t}^{\prime}>0$
\[
\left(  \mathcal{D}\left(  y_{\mathbf{t}^{\prime},\eta}\right)  =1\right)
\wedge\left(  \mathcal{D}\left(  y_{t,\eta}\right)  =0,\forall t\in
\lbrack0;\mathbf{t}^{\prime}-1]\right)
\]
we have
\[
x(\mathbf{t})\in\mathcal{B}_{\rho}\left(  \mathcal{F}\right),
\]
for some $\mathbf{t}\in\lbrack\max\left\{  \left(  \mathbf{t}^{\prime}%
-\Delta\right)  ,0\right\}  ;\mathbf{t}^{\prime}]$. \\
\end{definition}

\begin{remark}
Definition above is a natural extension of the notion of diagnosability given in \cite{DeSDiB:automatica17} for DES (with finite set of states), along two directions. First, it applies to nonlinear systems, hence dynamical systems with infinite set of states. Second, it is given in an approximate sense. Motivation for addressing approximate diagnosability primarily stems from the fact that outputs of nonlinear systems in concrete applications are measured by sensors that introduce measurement errors. Consequently, it is not possible to detect in general, with arbitrary small precision if the state of the system is or is not within the set of faulty states. When accuracy $\rho=0$ definition above coincides with the one of \cite{DeSDiB:automatica17}, when rewritten for nonlinear systems.
\end{remark}

\section{Checking approximate diagnosability}\label{sec4}

The approach that we follow to check approximate diagnosability of $\Sigma$ is based on the use of formal methods and in particular, of symbolic models. Symbolic models are abstract descriptions of control systems where each state corresponds to an aggregate of continuous states and each label to an aggregate of control inputs \cite{paulo}. This section is organized as follows. In Subsection \ref{subs1} we review the notion of systems, approximate relations and
extend the notion of diagnosability of \cite{DeSDiB:automatica17} to metric symbolic systems and in an approximate sense. In Subsection \ref{subs2} we give the main result of the paper: after having approximated the nonlinear system $\Sigma$ through a symbolic model, we establish connections between approximate diagnosability of the original nonlinear system $\Sigma$ and approximate diagnosability of the obtained symbolic model; computational complexity of the approach taken is also discussed.

\subsection{Systems, approximate relations and exact diagnosability} \label{subs1}

We start by recalling the notion of systems that we use to approximate the nonlinear system $\Sigma$.

\begin{definition}
\label{systems}
\cite{paulo}
A system is a tuple
\[
S=(X,X_0,U,\rTo,Y,H),
\]
consisting of
\begin{itemize}
\item a set of states $X$,
\item a set of initial states $X_0 \subseteq X$,
\item a set of inputs $U$,
\item a transition relation $\rTo \subseteq X\times U\times X$,
\item a set of outputs $Y$ and,
\item an output function $H:X\rightarrow Y$.
\end{itemize}
\end{definition}

\noindent
A transition $(x,u,x^{\prime})\in\rTo$ of $S$ is denoted by $x\rTo^{u}x^{\prime}$.
The evolution of systems is captured by the notions of state, input and output runs. Given a sequence of transitions of $S$
\begin{equation}
\label{seqtrans}
x_0 \rTo^{u_0} x_1 \rTo^{u_1} \,{...}\, \rTo^{u_{l-1}} x_l
\end{equation}
with $x_0 \in X_0$, the sequences
\begin{eqnarray}
&& r_X: \, x_0 \, x_1 \, ... \, x_l,\notag\\
&& r_U: \, u_0 \, u_1 \, ... \, u_{l-1}, \label{inputrun}\\
&& r_Y: H(x_0) \, H(x_1) \, ... \, H(x_l), \label{outputrun}
\end{eqnarray}
are called a \textit{state run}, an \textit{input run} and an \textit{output run} of $S$, respectively. \\
The accessible part of a system $S$, denoted $\Ac(S)$, is the collection of states of $S$ that are reached by state runs of $S$. \\
System $S$ is said to be:
\begin{itemize}
\item \textit{symbolic}, if $\Ac(S)$ and $U$ are finite sets;
\item \textit{metric}, if $X$ is equipped with a metric $\mathbf{d}:X\times X\rightarrow\mathbb{R}_{0}^{+}$;
\item \textit{deterministic}, if for any $x\in X$ and $u\in U$ there exists at most one transition $x \rTo^u x^+$ and \textit{nondeterministic}, otherwise.
\end{itemize}
In order to provide approximations of $\Sigma$ we need to recall the following notions of approximate simulation and bisimulation relations.

\begin{definition}
\label{ASR}
\cite{AB-TAC07}
Consider a pair of metric systems
\begin{equation}
\label{sys2}
S_i=(X_i,X_{0,i},U_i,\rTo_i,Y_i,H_i),
\end{equation}
with $X_1$ and $X_2$ subsets of some metric set $X$ equipped with metric $\mathbf{d}$, and let $\varepsilon\in\mathbb{R}^{+}_{0}$ be a given accuracy. Consider a relation
\begin{equation}
\label{rel}
\mathcal{R}\subseteq X_{1}\times X_{2}
\end{equation}
satisfying the following conditions:
\begin{itemize}
\item
[(i)] $\forall x_{1}\in X_{0,1}$ $\exists x_{2}\in X_{0,2}$ such that $(x_{1},x_{2})\in \mathcal{R}$;
\item [(ii)] $\forall (x_{1},x_{2})\in \mathcal{R}$, $\mathbf{d}(x_{1},x_{2}) \leq \varepsilon$.
\end{itemize}

Relation $\mathcal{R}$ is an \emph{$\varepsilon$-approximate ($\varepsilon$A) simulation relation} from $S_{1}$ to $S_{2}$ if it enjoys conditions (i)--(ii) and the following one:
\begin{itemize}
\item[(iii)] $\forall (x_{1},x_{2})\in \mathcal{R}$ if \mbox{$x_{1}\rTo_{1}^{u_{1}}x'_{1}$} then there exists \mbox{$x_{2}\rTo_{2}^{u_{2}}x'_{2}$}  such that $(x^{\prime}_{1},x^{\prime}_{2})\in \mathcal{R}$.
\end{itemize}
System $S_{1}$ is $\varepsilon$-simulated by $S_{2}$, if there exists an $\varepsilon$-approximate simulation relation from $S_{1}$ to $S_{2}$.\\
Relation $\mathcal{R}$ in (\ref{rel}) is an \emph{$\varepsilon$-approximate ($\varepsilon$A) bisimulation relation} between $S_{1}$ and $S_{2}$ if
\begin{itemize}
\item $\mathcal{R}$ is an $\varepsilon$A simulation relation from $S_{1}$ to $S_{2}$, and
\item $\mathcal{R}^{-1}$ is an $\varepsilon$A simulation relation from $S_{2}$ to $S_{1}$.
\end{itemize}
Systems $S_{1}$ and $S_{2}$ are $\varepsilon$-bisimilar, denoted $S_1 \cong_{\varepsilon} S_2$, if there exists an $\varepsilon$-approximate bisimulation relation between $S_{1}$ and $S_{2}$.
\end{definition}

We conclude this section with the following

\begin{definition}
(Approximate diagnosability of metric systems)
Consider a metric system $S=(X,X_0,U,\rTo,Y,H)$, with metric $\widehat{\mathbf{d}}$, and a set $\widehat{\mathcal{F}}\subseteq X$ of faulty states with
\[
\widehat{\mathcal{F}} \cap X_{0}=\varnothing.
\]
Denote by $\mathbf{X}$ and $\mathbf{Y}$ the collection of state runs and of output runs of $S$, respectively. Denote by $\widehat{\mathcal{B}}_{\widehat{\rho}}(\widehat{x})$ the closed ball induced by metric $\widehat{\mathbf{d}}$ centered at $\widehat{x}\in X$ and with radius $\widehat{\rho}$, i.e.
\[
\widehat{\mathcal{B}}_{\widehat{\rho}}(\widehat{x})=
\{
x\in X | \widehat{\mathbf{d}}(x,\widehat{x})\leq \widehat{\rho}
\}.
\]
Given $\widehat{X}\subseteq X$, denote by $\widehat{\mathcal{B}}_{\widehat{\rho}}(\widehat{X})$ the set
\[
\bigcup_{\widehat{x}\in \widehat{X}}\widehat{\mathcal{B}}_{\widehat{\rho}}(\widehat{x}).
\]
Given a desired accuracy $\widehat{\rho}\in\mathbb{R}^{+}_0$, system $S$ is $(\widehat{\rho},\widehat{\mathcal{F}})$--diagnosable if there exists a finite delay
$\widehat{\Delta}\in\mathbb{N}$, and
\[
\widehat{\mathcal{D}}:\mathbf{Y}\rightarrow\left\{  0,1\right\}  ,
\]
called the diagnoser, with
\[
\widehat{\mathcal{D}}
\left(  y_{0}\right)  =0,
\]
where $y_0$ is any output run of $S$ with length $1$, such that for any state run $x_0 x_1 \dots x_{\mathbf{t}}\in\mathbf{X}$
and corresponding output run $y_0 y_1 \dots y_{\mathbf{t}}\in\mathbf{Y}$ of $S$, whenever for some $\mathbf{t}>0$
\[
\left(  x_{\mathbf{t}}\in \widehat{\mathcal{F}}\right)  \wedge\left(  x_t\notin \widehat{\mathcal{F}},\forall
t\in\lbrack0;\mathbf{t}-1]\right),
\]
we have
\[
\widehat{\mathcal{D}} \left(  y_{\mathbf{t}+\widehat{\Delta}}\right)  =1.
\]
Conversely, whenever for some $\mathbf{t}^{\prime}>0$
\[
\left(  \widehat{\mathcal{D}}\left(  y_{\mathbf{t}^{\prime}}\right)  =1\right)
\wedge\left(  \widehat{\mathcal{D}}\left(  y_{t}\right)  =0,\forall t\in
\lbrack0;\mathbf{t}^{\prime}-1]\right)
\]
we have
\[
x_{\mathbf{t}} \in \widehat{\mathcal{B}}_{\widehat{\rho}}(\widehat{F}),
\]
for some $\mathbf{t}\in\lbrack\max\{(  \mathbf{t}^{\prime}-\widehat{\Delta}),0\}  ;\mathbf{t}^{\prime}]$. \\
\end{definition}

\begin{remark} \label{RemComplexDis}
Definition above extends the notion of diagnosability of \cite{DeSDiB:automatica17}, to metric systems in the sense of Definition \ref{systems}. 
In \cite{DeSDiB:automatica17}, diagnosability with accuracy $\widehat{\rho}=0$ of DES has been
characterized in a set membership framework, and it is shown that both
space and time computational complexities in checking this property are polynomial in the cardinality of the set of states of the
DES. Since, the conditions of \cite{DeSDiB:automatica17} rely upon the set $\widehat{\mathcal{F}}$ and its complement (see Theorem 20), the generalization of such conditions to $(\widehat{\rho},\widehat{\mathcal{F}})$--diagnosability with $\widehat{\rho}\geq0$ can
simply be obtained by replacing the complement of $\widehat{\mathcal{F}}$ with the
complement of $\widehat{\mathcal{B}}_{\widehat{\rho}}(\widehat{\mathcal{F}})$; obtained algorithms remain of polynomial computational complexity.
\end{remark}

\subsection{Main result}\label{subs2}

We start by defining a symbolic system that approximates $\Sigma$ in the sense of approximate bisimulation for any desired accuracy. For convenience, we reformulate the nonlinear system $\Sigma$ by using the formalism given in Definition \ref{systems}.
\begin{definition}
\label{syssigma}
Given $\Sigma$, define the system
\[
S(\Sigma)=(X,X_0,U,\rTo,X_m,Y,H),
\]
where
\begin{itemize}
\item $X=\mathbb{R}^n$;
\item $X_0=\mathcal{X}_0$;
\item $U$ coincides with the set $U$ in (\ref{plant});
\item $x \rTo^{u} x^+$, if $x^{+}=f(x,u)$;
\item $Y=\mathbb{R}^p$;
\item $H(x)=
\left[
\begin{array}
{cc}
I_p & 0
\end{array}
\right]x$, for all $x\in X$.
\end{itemize}
\end{definition}

System $S(\Sigma)$ is metric because set $X=\mathbb{R}^n$ can be equipped with a metric $\mathbf{d}$; in the sequel we choose metric
\begin{equation}
\label{metric}
\mathbf{d}(x,x')=\Vert x- x' \Vert, x,x'\in \mathbb{R}^n.
\end{equation}
System $S(\Sigma)$ will be approximated by means of a system that we now introduce.

\begin{definition}
\label{symbmod}
Given $\Sigma$, a state and output quantization parameter $\eta \in\mathbb{R}^{+}$ and an input quantization parameter $\mu \in\mathbb{R}^{+}$, define the system
\[
S_{\eta,\mu}(\Sigma)=(X_{\eta,\mu},X_{\eta,\mu,0},U_{\eta,\mu},\rTo_{\eta,\mu},Y_{\eta,\mu},H_{\eta,\mu}),
\]
where:
\begin{itemize}
\item $X_{\eta,\mu}=[\mathbb{R}^n]^n_{\eta}$;
\item $X_{\eta,\mu,0}=[\mathcal{X}_0]^n_{\eta}$;
\item $U_{\eta,\mu}=[U]^m_{\mu}$;
\item $\xi \rTo^{v}_{\eta,\mu} \xi^{+}$, if $\xi^{+}=[f(\xi,v)]^n_{\eta}$;
\item $Y_{\eta,\mu}=[\mathbb{R}^{p}]^p_{\eta}$;
\item $H_{\eta,\mu}(\xi)=\left[
\begin{array}
{cc}
I_p & 0
\end{array}
\right]\xi$, for all $\xi\in X_{\eta,\mu}$.
\end{itemize}
\end{definition}

The basic idea in the construction above is to replace each state $x$ in $\Sigma$ by its quantized value $\xi = [x]^n_{\eta}$ and each input $u\in U$ by its quantized value $v = [u]^m_{\mu}$ in $S_{\eta,\mu}(\Sigma)$. Accordingly, evolution of system $\Sigma$ with initial state $x$ and input $v$ to state $x^+=f(\xi,v)$, is captured by the transition $\xi {\rTo^{v}_{\eta}} \xi^+$ in system $S_{\eta,\mu}(\Sigma)$, where $\xi$ and $\xi^+$ are the quantized values of $x$ and $x^+$, respectively, i.e., $\xi = [x]^n_{\eta}$ and $\xi^+ = [x^+]^n_{\eta}$. System $S_{\eta,\mu}(\Sigma)$ is metric; in the sequel we use the metric $\mathbf{d}$ in (\ref{metric}); this choice is allowed because $X_{\eta,\mu}\subset X$. Moreover, by definition of the transition relation $\rTo_{\eta,\mu}$, system $S_{\eta,\mu}(\Sigma)$ is deterministic. By definition of $X_{\eta,\mu}$ and $U_{\eta,\mu}$, system $S_{\eta,\mu}(\Sigma)$ is countable. We now consider the following
\begin{assumption}
\label{A1}
A locally Lipschitz function
\[
V: \mathbb{R}^{n} \times \mathbb{R}^{n} \rightarrow \mathbb{R}^{+}_{0}
\]
exists for nonlinear system $\Sigma$, which satisfies the following inequalities for some $\mathcal{K}_{\infty}$ functions $\underline{\alpha}$, $\overline{\alpha}$, $\lambda$ and $\mathcal{K}$ function $\sigma$:\\
(i) $\underline{\alpha}(\left\Vert x-x' \right\Vert )\leq V(x,x')\leq \overline{\alpha}(\left\Vert x-x' \right\Vert )$, for any $x,x'\in\mathbb{R}^{n}$; \\
(ii) $V(f(x,u),f(x',u')) - V(x,x') \leq -\lambda ( V(x,x') ) + \sigma (\left\Vert u-u' \right\Vert )$, for any $x,x'\in\mathbb{R}^{n}$ and any
$u,u'\in U$.\\
\end{assumption}
Function $V$ is called an incremental input--to--state stable ($\delta$--ISS) Lyapunov function \cite{IncrementalS,BayerECC2013} for nonlinear system $\Sigma$. Assumption \ref{A1} has been shown in \cite{BayerECC2013} to be a sufficient condition for $\Sigma$ to fulfill the $\delta$--ISS stability property \cite{IncrementalS,BayerECC2013}. \\
We now have all the ingredients to present the following

\begin{proposition}
\label{ThSCC}
Suppose that Assumption \ref{A1} holds and let $L$ be a Lipschitz constant of function $V$ in $\mathbb{R}^n \times \mathbb{R}^n$.
Then, for any desired accuracy $\varepsilon\in\mathbb{R}^{+}$ and for any quantization parameters $\eta,\mu \in \mathbb{R}^{+}$ satisfying the following inequalities
\begin{equation}
\begin{array}
{l}
L \eta  + \sigma(\mu) \leq (\lambda \circ \underline{\alpha})(\varepsilon),\\
\overline{\alpha}(\eta)\leq \underline{\alpha}(\varepsilon),\\
\end{array}
\label{statem}
\end{equation}
relation $\mathcal{R}_{\varepsilon}\subseteq \mathbb{R}^n \times X_{\eta,\mu}$ specified by
\begin{equation}
\label{relinit}
(x,\xi)\in \mathcal{R}_{\varepsilon} \Leftrightarrow V(x,\xi) \leq \underline{\alpha}(\varepsilon)
\end{equation}
is an $\varepsilon$--approximate bisimulation between $S(\Sigma)$ and $S_{\eta,\mu}(\Sigma)$. Consequently, systems $S(\Sigma)$ and $S_{\eta,\mu}(\Sigma)$ are $\varepsilon$--bisimilar.
\end{proposition}

\begin{proof}
Direct consequence of Proposition 1 in \cite{PolaTAC16}.
\end{proof}

\medskip

We now show that under Assumption \ref{A1}, system $S_{\eta,\mu}(\Sigma)$ is not only countable but also symbolic.

\begin{proposition}
Suppose that Assumption \ref{A1} holds. Then, for any quantization parameters $\eta,\mu\in \mathbb{R}^+$, system $S_{\eta,\mu}(\Sigma)$ is symbolic.
\end{proposition}

\begin{proof}
Under Assumption \ref{A1}, the nonlinear system $\Sigma$ is $\delta$--ISS \cite{BayerECC2013}.
Given $\eta, \mu \in \mathbb{R}^+$, select $\varepsilon\in\mathbb{R}^+$ satisfying the inequality in (\ref{statem}).
By Proposition \ref{ThSCC}, for any $t_f \in \mathbb{N}$ and for any state run
\[
\xi(0) \, \xi(1) \,  ...  \, \xi(t_f)
\]
of $S_{\eta,\mu}(\Sigma)$ there exists a state trajectory $x:[0;t_f] \rightarrow \mathbb{R}^n$ of $\Sigma$ such that:
\[
(\xi(t),x(t))\in \mathcal{R}_{\varepsilon}, \forall t \in [0;t_f].
\]
By definition of $\mathcal{R}_{\varepsilon}$ in (\ref{relinit}), the $\delta$--ISS property and Assumption \ref{A0}, we get for any $t\in [0;t_f]$ and for any $t_f \in \mathbb{N}$:
\[
\begin{array}
{rcl}
\Vert \xi(t) \Vert & \leq & \Vert \xi(t) - x(t) \Vert + \Vert x(t) \Vert \\
                   & \leq & \varepsilon + \beta(\Vert x(0) \Vert,t) + \gamma(\sup_{t'\in [0;t]}\Vert u(t') \Vert)\\
								 	 & \leq & \varepsilon + \beta(\Vert x(0) \Vert,0) + \gamma(\sup_{t'\in [0;t]}\Vert u(t') \Vert)\\
									 & \leq & \varepsilon + \max_{x_0 \in \mathcal{X}_0} \beta(\Vert x(0) \Vert,0) +\gamma (\max_{u \in U} \Vert u \Vert),
\end{array}
\]
for some $\mathcal{KL}$ function $\beta$ and $\mathcal{K}$ function $\gamma$ \cite{BayerECC2013}. Hence,
\[
\xi(t)\in \mathcal{B}_{r}(0),\, t\in \mathbb{N},
\]
with $r=\varepsilon + \max_{x_0 \in \mathcal{X}_0} \beta(\Vert x(0) \Vert,0) +\gamma (\max_{u \in U} \Vert u \Vert)$, and since $\xi(t)\in X_{\eta,\mu}$ for all $t\in \mathbb{N}$, we get
\[
\Ac(S_{\eta,\mu}(\Sigma)) \subseteq \mathcal{B}_{r}(0) \cap X_{\eta,\mu}
\]
that is a finite set. Finally, since $U$ is compact then $[U]_{\mu}^m$ is finite from which, the result follows.
\end{proof}

Computational complexity in constructing $S_{\eta,\mu}(\Sigma)$ is discussed in the following

\begin{proposition}
\label{complexsymb}
Space and time computational complexities in computing $S_{\eta,\mu}(\Sigma)$ are exponential with the dimension $n$ of state space and with the dimension $m$ of the input space of $\Sigma$.
\end{proposition}

One way to mitigate computational complexity above is in constructing only the accessible part of $S_{\eta,\mu}(\Sigma)$; similar ideas were explored in \cite{PolaTAC12} by following on--the--fly techniques studied in e.g. \cite{onthefly3,onthefly2} for efficient formal verification and control design of transition systems.\\
We now have all the ingredients to present the main result of this paper establishing connections between approximate diagnosability of $S_{\eta,\mu}(\Sigma)$ and approximate diagnosability of the original nonlinear system $\Sigma$.\\
Given a set $\mathcal{F}\subseteq \mathbb{R}^n$ and an accuracy $\varepsilon\in\mathbb{R}^+$, consider the sets
\[
\begin{array}
{l}
\mathcal{F}_{\varepsilon}=\mathcal{B}_{\varepsilon}\left(  \mathcal{F}\right)  \cap\left[  \mathbb{R}^{n}\right]_{\eta}^{n},\\
\mathcal{F}'_{\varepsilon}=\{x\in \mathcal{F}: \mathcal{B}_{\varepsilon}(x)\subseteq \mathcal{F}\}\cap\left[  \mathbb{R}^{n}\right]_{\eta}^{n}.
\end{array}
\]
By construction above, we get:
\[
\mathcal{F}'_{\varepsilon} \subseteq (\mathcal{F} \cap\left[  \mathbb{R}^{n}\right]_{\eta}^{n}) \subseteq \mathcal{F}_{\varepsilon}.
\]

\begin{theorem}
\label{mainth1}
Consider nonlinear system $\Sigma$ in (\ref{plant}) satisfying Assumption \ref{A1} and a set $\mathcal{F}\subseteq \mathbb{R}^n$. Consider a triplet $\varepsilon,\eta,\mu\in\mathbb{R}^+$ of parameters satisfying (\ref{statem}). The following statements hold:
\begin{itemize}
\item [i)] If $S_{\eta,\mu}\left(
\Sigma\right)  $ is $\left(  k\eta,\mathcal{F}_{\varepsilon}\right)$--diagnosable, for some
$k\in \mathbb{N}$, then $\Sigma$ is $(\rho,\mathcal{F})$--diagnosable, for any
\[
\rho>2\varepsilon+k\eta.
\]
\item [ii)] Suppose that set $\mathcal{F}$ is with interior and parameter $\varepsilon\in\mathbb{R}^+$ is such that\footnote{Since $\mathcal{F}$ is with interior there always exists $\varepsilon\in\mathbb{R}^+$ satisfying (\ref{condeps}).}
\begin{equation}
\label{condeps}
\mathcal{F}_{\varepsilon}^{\prime}\neq\varnothing.
\end{equation}
If $\Sigma$ is $(\rho,\mathcal{F})$--diagnosable, for some $\rho \in\mathbb{R}^+_0$, then $S_{\eta,\mu}\left(
\Sigma\right)  $ is $\left(  k^{\prime}\eta,\mathcal{F}_{\varepsilon}^{\prime
}\right)$--diagnosable, for any integer
\[
k^{\prime}>\min \{h\in\mathbb{N}:\left(  \rho+2\varepsilon\right)  \leq h\eta \}.
\]
\end{itemize}
\end{theorem}

\begin{proof}
(Proof of i).)
By contradiction, suppose that $S_{\eta,\mu}\left(  \Sigma\right)  $ is
$\left(  k\eta,\mathcal{F}_{\varepsilon}\right)$--diagnosable, but $\Sigma$ is not
$(\rho,\mathcal{F})$--diagnosable, with $\rho>2\varepsilon+k\eta$. Then
$\forall\Delta\in\mathbb{N}$ there exists a state trajectory $x^{f}$ of
$\Sigma$ such that for some $\mathbf{t}>0$
\[
\left(  x^{f}(\mathbf{t})\in\mathcal{F}\right)  \wedge\left(  x(t)\notin%
\mathcal{F},\forall t\in\lbrack0;\mathbf{t}-1]\right)
\]
and a state trajectory $x^{s}$ of $\Sigma$ such that
\begin{equation}
\mathcal{B}_{\rho}\left(  x^{s}(t)\right)  \cap\mathcal{F}=\varnothing,\forall
t\in\lbrack0;\mathbf{t}+\Delta]\label{uno}%
\end{equation}
and, by denoting by $y_{\mathbf{t}+\Delta,\eta}^{f}$ and $y_{\mathbf{t}%
+\Delta,\eta}^{s}$ the quantized output trajectories associated to $\left.  x^{f}%
\right\vert _{\left[  0;\mathbf{t}+\Delta\right]  }$ and to $\left.
x^{s}\right\vert _{\left[  0;\mathbf{t}+\Delta\right]  }$, respectively, we
have
\[
y_{\mathbf{t}+\Delta,\eta}^{f}=y_{\mathbf{t}+\Delta,\eta}^{s}.
\]
Since $S(\Sigma)\cong_{\varepsilon} S_{\eta,\mu}(\Sigma)$, for any state trajectory $x$ of $\Sigma$ there exists a state run $\xi(0)\, \xi(1)\, \dots$ of $S_{\eta,\mu}\left(  \Sigma\right)  $, such that
\[
\left\Vert \xi(t)-x(t)\right\Vert \leq\varepsilon,\forall t\in\mathbb{N}%
\]
By construction of $\mathcal{F}_{\varepsilon}$, if $x^{f}(\mathbf{t})\in\mathcal{F}$
then 
\[
\mathcal{B}_{\varepsilon}\left(  x^{f}(\mathbf{t})\right)  \cap\left[
\mathbb{R}^{n}\right]  _{\eta}^{n}\subseteq\mathcal{F}_{\varepsilon}.
\]
Moreover,
since $\rho>2\varepsilon+k\eta$ and condition (\ref{uno}) holds, then 
\[
\mathcal{B}_{\varepsilon+k\eta}\left(  x^{s}(t)\right)
\cap\mathcal{F}_{\varepsilon}=\varnothing. 
\]
Therefore $\forall\Delta\in\mathbb{N}$
there exist two state runs $\xi^{\prime}(0)\, \xi^{\prime}(1)\, \dots$ and $\xi"(0)\, \xi"(1)\, \dots$ of $S_{\eta,\mu}\left(  \Sigma\right)  $, the first one such that for some $\mathbf{t}^{\prime}\in\lbrack0;\mathbf{t}]$
\[
\left(  \xi^{\prime}(\mathbf{t}^{\prime})\in\mathcal{F}_{\varepsilon}\right)
\wedge\left(  \left(  \mathbf{t}^{\prime}=0\right)  \vee\left(  \xi^{\prime
}(t)\notin\mathcal{F}_{\varepsilon},\forall t\in\lbrack0;\mathbf{t}^{\prime
}-1]\right)  \right)
\]
and the other one such that
\[
\xi"(t)\notin\mathcal{B}_{k\eta}\left(  \mathcal{F}_{\varepsilon}\right)  ,\forall
t\in\lbrack0;\mathbf{t}^{\prime}+\Delta]
\]
with the same corresponding output runs, i.e.
\[
y_{\mathbf{t}+\Delta,\eta}^{\prime}=y"_{\mathbf{t}+\Delta,\eta}%
\]
Therefore $S_{\eta,\mu}\left(  \Sigma\right)  $ is not $\left(  k\eta
,\mathcal{F}_{\varepsilon}\right)$--diagnosable, and the first statement  follows. \\
(Proof of ii).) Again by contradiction, suppose that $\Sigma$ is $(\rho,\mathcal{F}
)$--diagnosable, but  $S_{\eta,\mu}\left(  \Sigma\right)  $ is not  $\left(
k^{\prime}\eta,\mathcal{F}_{\varepsilon}^{\prime}\right)$--diagnosable, with
$\varepsilon,\eta,\mu$ satisfying (\ref{statem}), $\varepsilon$ such that
$\mathcal{F}_{\varepsilon}^{\prime}\neq\varnothing$, and $k^{\prime}>\min
_{h\in\mathbb{N}}h\eta:\left(  \rho+2\varepsilon\right)  \leq h\eta$. Then
$\forall\Delta\in\mathbb{N}$ there exists a state run $\xi^{f}$ of
$S_{\eta,\mu}\left(  \Sigma\right)  $ such that for some $\mathbf{t}>0$
\[
\left(  \xi^{f}(\mathbf{t})\in\mathcal{F}_{\varepsilon}^{\prime}\right)
\wedge\left(  x(t)\notin\mathcal{F}_{\varepsilon}^{\prime},\forall t\in
\lbrack0;\mathbf{t}-1]\right)
\]
and a state run $\xi^{s}$ of $S_{\eta,\mu}\left(  \Sigma\right)  $ such that
\begin{equation}
\mathcal{B}_{k^{\prime}\eta}\left(  \xi^{s}(t)\right)  \cap\mathcal{F}
_{\varepsilon}^{\prime}=\varnothing,\forall t\in\lbrack0;\mathbf{t}%
+\Delta]\label{due}%
\end{equation}
and, by denoting with $y_{\mathbf{t}+\Delta,\eta}^{f}$ and $y_{\mathbf{t}%
+\Delta,\eta}^{s}$ the output runs associated to $\left.  \xi
^{f}\right\vert _{\left[  0;\mathbf{t}+\Delta\right]  }$ and to $\left.
\xi^{s}\right\vert _{\left[  0;\mathbf{t}+\Delta\right]  }$, respectively, we
have
\[
y_{\mathbf{t}+\Delta,\eta}^{f}=y_{\mathbf{t}+\Delta,\eta}^{s}.
\]
Since $S(\Sigma)\cong_{\varepsilon} S_{\eta,\mu}(\Sigma)$, for any state run $\xi(0)\, \xi(1)\, \dots$ of $S_{\eta,\mu}\left(  \Sigma\right)  $ there exists a state trajectory $x$ of $\Sigma$, such that
\[
\left\Vert \xi(t)-x(t)\right\Vert \leq\varepsilon,\forall t\in\mathbb{N}.
\]
By construction of $\mathcal{F}_{\varepsilon}^{\prime}$, if $\xi^{f}(\mathbf{t}
)\in\mathcal{F}_{\varepsilon}^{\prime}$ then 
\[
\mathcal{B}_{\varepsilon}\left(
\xi^{f}(\mathbf{t})\right)  \subseteq\mathcal{F}.
\]
Moreover, since $k^{\prime}
>\min_{h\in\mathbb{N}}h\eta:\left(  \rho+2\varepsilon\right)  \leq h\eta$, and
since condition $\left(  \ref{due}\right)  $ holds, then $\mathcal{B}
_{k^{\prime}\eta}\left(  \xi^{s}(t)\right)  \cap\mathcal{F}=\varnothing$. Therefore
$\forall\Delta\in\mathbb{N}$ there exist two state trajectories $x^{\prime}$ and $x"$ of $\Sigma$,
the first one such that for some $\mathbf{t}^{\prime}\in\lbrack0;\mathbf{t}]$
\[
\left(  x^{\prime}(\mathbf{t}^{\prime})\in\mathcal{F}\right)  \wedge\left(
x^{\prime}(t)\notin\mathcal{F},\forall t\in\lbrack0;\mathbf{t}^{\prime}-1]\right)
\]
and the other one such that
\[
x"(t)\notin\mathcal{B}_{k^{\prime}\eta}\left(  \mathcal{F}\right)  ,\forall
t\in\lbrack0;\mathbf{t}^{\prime}+\Delta]
\]
with the same corresponding quantized output trajectories, i.e.
\[
y_{\mathbf{t}+\Delta,\eta}^{\prime}=y"_{\mathbf{t}+\Delta,\eta}.
\]
Therefore $\Sigma$ is not $\left(  k^{\prime}\eta,\mathcal{F}\right)$--diagnosable. Since
$k^{\prime}\eta\geq\left(  \rho+2\varepsilon\right)  $, then $k^{\prime}%
\eta>\rho$, $\Sigma$ is not $\left(  \rho,\mathcal{F}\right)$--diagnosable and the proof
is complete.
\end{proof}

\begin{remark}
While statement i) of Theorem \ref{mainth1} is useful to check if $\Sigma$ is $(\rho,\mathcal{F})$--diagnosable, statement ii) can be used in its logical negation form as a tool to check if $\Sigma$ is not $(\rho,\mathcal{F})$--diagnosable.
\end{remark}

We conclude this section with a computational complexity analysis of the approach proposed. By combining Remark \ref{RemComplexDis}  and Proposition
\ref{complexsymb} we get

\begin{theorem}
Space and time computational complexities in checking $(\rho,\mathcal{F})$--diagnosability of $\Sigma$ are exponential with the dimension $n$ of state space and with the dimension $m$ of the input space of $\Sigma$.
\end{theorem}

\section{Conclusions}\label{sec5}
In this paper we proposed a novel notion of diagnosability, termed approximate diagnosability, for discrete--time nonlinear systems with unknown inputs and quantized output measurements. Under an assumption of incremental stability of the nonlinear system we first derived a symbolic model. We extended the classical notion of diagnosability given for DES to metric symbolic systems. We then established the relation between approximate diagnosability of the nonlinear system and approximate diagnosability of the symbolic model. Computational complexity of the approach taken is also discussed.

\bibliographystyle{plain}
\bibliography{biblio2,biblio1}

\end{document}